\documentclass[11pt]{article}
\usepackage[a4paper,margin=1.5cm]{geometry}
\usepackage{setspace}

\usepackage{amstext}
\usepackage{amsthm}
\usepackage{amsopn}
\usepackage{amsfonts}
\usepackage{amsmath}
\usepackage{amssymb}
\usepackage{framed}
\usepackage{enumitem}

\usepackage{comment}

\usepackage[dvipsnames]{xcolor}
\definecolor{greytext}{gray}{0.5}

\usepackage{url}
\usepackage[numbers,sort&compress]{natbib}

\renewcommand\UrlFont{\color{black!80}}

\DeclareUrlCommand\DOI{}

\usepackage{microtype}
\usepackage{setspace}
\usepackage[bookmarks,bookmarksnumbered,
    colorlinks,
    linktocpage
    ]{hyperref}

\usepackage[capitalise]{cleveref}

\newcommand\theoremlevel{section}
\usepackage[dvipsnames]{xcolor}
\definecolor{greytext}{gray}{0.5}

\usepackage{amstext}
\usepackage{amsthm}
\usepackage{amsopn}
\usepackage{amsfonts}
\usepackage{amsmath}
\usepackage{amssymb}
\usepackage{stackrel}
\usepackage{framed}

\usepackage{array}
\usepackage{dsfont}

\usepackage{mathrsfs}
\usepackage{mathtools}
\usepackage{enumitem}
\usepackage{tikz-cd}
\usepackage{arydshln}
\usepackage{subcaption}

\usepackage[scr = euler]{mathalpha}

\usepackage{chngcntr}
\counterwithin*{equation}{section}
\counterwithin{figure}{section}

\makeatletter
\@ifundefined{theoremlevel}{%
    \newtheorem{theorem}{Theorem}%
}{%
    \newtheorem{theorem}{Theorem}[\theoremlevel]%
}

\newtheorem{lemma}[theorem]{Lemma}
\newtheorem{proposition}[theorem]{Proposition}

\theoremstyle{definition}
\newtheorem{definition}[theorem]{Definition}
\newtheorem{example}[theorem]{Example}

\theoremstyle{remark}
\newtheorem{remark}[theorem]{Remark}

\let\oldcite\cite
\renewcommand\cite[2][]{{\rm\oldcite[#1]{#2}}}

\numberwithin{equation}{section}

\def\(#1\){$\displaystyle#1$}
\def\[#1\]{\begin{align*}#1\end{align*}}

\newcommand\sC{\mathscr{C}}
\newcommand\sD{\mathscr{D}}

\newcommand\sK{\mathscr{K}}

\newcommand\sS{\mathscr{S}}
\newcommand\sV{\mathscr{V}}

\newcommand\Hom{\operatorname{Hom}}

\newcommand\varlim{\varprojlim}
\newcommand\varcolim{\varinjlim}

\makeatletter
\newcommand{\pto}{}
\newcommand{\pgets}{}

\DeclareRobustCommand{\pto}{\mathrel{\mathpalette\p@to@gets\to}}
\DeclareRobustCommand{\pgets}{\mathrel{\mathpalette\p@to@gets\gets}}

\newcommand{\p@to@gets}[2]{%
  \ooalign{\hidewidth$\m@th#1\mapstochar\mkern5mu$\hidewidth\cr$\m@th#1\to$\cr}%
}
\makeatother

\newcommand\Cat{\mathbf{Cat}}

\newcommand\Grpd{\mathbf{Grpd}}

\newcommand\Pres{\mathbf{Pres}}

\newcommand\SymMon{\mathbf{SymMon}}

\tikzset{Rightarrow/.style={double equal sign distance,>={Implies},->},
Rrightarrow/.style={-,preaction={draw,Rightarrow}},
Rrrightarrow/.style={preaction={draw,Rightarrow,shorten >=0pt},shorten >=1pt,-,double,double
distance=0.2pt}}

\tikzcdset{
  cells={font=\everymath\expandafter{\the\everymath\displaystyle}},
}

\newcommand{\manref}[2]{\hyperref[#1]{#2~\ref*{#1}}}

\newcommand{\rank}{\operatorname{rank}}

\usepackage{datetime}
\title{A note on Noetherian $(\infty, \infty)$-categories}
\date{\monthname~\the\year}
\author{Zach Goldthorpe}

\begin{document}
\maketitle
\begin{abstract}
    The purpose of this note is to resolve \cite[Conjecture 3.4.3]{goldthorpe:fixed}, regarding the initial algebra for the enrichment endofunctor $(-)\Cat$ over general symmetric monoidal $(\infty, 1)$-categories.
    We prove that Ad\'amek's construction of an initial algebra for $(-)\Cat$ does not terminate; more precisely, we show that Ad\'amek's construction of an initial algebra for the endofunctor $(-)\Cat^{<\lambda}$ that sends a symmetric monoidal $(\infty, 1)$-category $\sV$ to the $(\infty, 1)$-category of $\sV$-enriched categories with at most $\lambda$ equivalence classes of objects terminates in precisely $\lambda$ steps.
    We also prove that an initial algebra for the endofunctor $(-)\Cat$ exists nonetheless, and characterise it as the $(\infty, 1)$-category consisting of those $(\infty, \infty)$-categories that satisfy a weak finiteness property we call Noetherian.
\end{abstract}

\tableofcontents{}

\section{Preliminaries}

Throughout this note, an ``$\infty$-category'' refers to an $(\infty, 1)$-category.

The theory of enrichment in a general monoidal $\infty$-category was developed extensively in \cite{gepner-haugseng}.
If $\SymMon_\infty$ denotes the $\infty$-category of large symmetric monoidal $\infty$-categories, then enrichment induces an endofunctor $(-)\Cat : \SymMon_\infty \to \SymMon_\infty$.
If $\SymMon_\infty^\Pres$ denotes the $\infty$-category of presentably symmetric monoidal $\infty$-categories (and symmetric monoidal left adjoints between them), then enrichment restricts further to an endofunctor on $\SymMon_\infty^\Pres$ as well.

We can then use this endofunctor to construct an $\infty$-category of $(\infty, r)$-categories for all finite $r\geq0$ by induction:
\[
    \Cat_{(\infty, 0)} &:= \Grpd_\infty := \sS, & \Cat_{(\infty, r+1)} &:= (\Cat_{(\infty, r)})\Cat
\]
In the resulting tower of inclusions
$$
\Cat_{(\infty, 0)} \subseteq \Cat_{(\infty, 1)} \subseteq \Cat_{(\infty, 2)} \subseteq\cdots
$$
every inclusion $\Cat_{(\infty, r)}\subseteq\Cat_{(\infty, r+1)}$ admits a left adjoint $\pi_{\leq r}$ and a right adjoint $\kappa_{\leq r}$.
These tower of adjoints define two natural constructions of the $\infty$-category of $(\infty, \infty)$-categories:
\[
    \Cat_{(\infty, \infty)} &:= \varlim\left(\dots \to \Cat_{(\infty, 2)} \xrightarrow{\kappa_{\leq1}} \Cat_{(\infty, 1)} \xrightarrow{\kappa_{\leq0}} \Cat_{(\infty, 0)}\right) \\
    \Cat_\omega &:= \varlim\left(\dots \to \Cat_{(\infty, 2)} \xrightarrow{\pi_{\leq1}} \Cat_{(\infty, 1)} \xrightarrow{\pi_{\leq0}} \Cat_{(\infty, 0)}\right)
\]
where $\Cat_{(\infty, \infty)}$ consists of higher categories wherein equivalences are generated inductively, and $\Cat_\omega$ is the full subcategory of $\Cat_{(\infty, \infty)}$ spanned by those $(\infty, \infty)$-categories wherein the equivalences are weakly coinductive; see \cite[Remark 3.0.5]{goldthorpe:fixed}.

The main result of \cite{goldthorpe:fixed} is that both $\Cat_\omega$ and $\Cat_{(\infty, \infty)}$ define universal fixed points with respect to enrichment.
More precisely, in the $\infty$-category of pairs $(\sV, \tau)$, where $\sV$ is a presentably symmetric monoidal $\infty$-category, and $\tau : \sV \xrightarrow\sim \sV\Cat$ is a symmetric monoidal equivalence, and the morphisms between them being symmetric monoidal left adjoints that preserve $\tau$, then $\Cat_\omega$ defines a terminal object, and $\Cat_{(\infty, \infty)}$ defines an initial object.

If we consider the larger category of pairs $(\sV, \tau)$ where $\sV$ is only symmetric monoidal and not necessarily presentably so, and the morphisms between them are not necessarily left adjoints, then $\Cat_\omega$ continues to define a terminal object.
However, the presentability assumption is crucial for $\Cat_{(\infty, \infty)}$.

The universal properties are derived using an $\infty$-categorical extension of the theory of endofunctor algebras developed in \cite[\S2]{goldthorpe:fixed}; see \cite{adamek:modern} for the classical theory.
For an $\infty$-category $\sK$ with an endofunctor $F:\sK\to\sK$, an \emph{$F$-algebra} is a pair $(A, \mu)$ where $A\in\sK$ and $\mu : FA \to A$ is an arbitrary morphism.
Dually, an \emph{$F$-coalgebra} is a pair $(C, \nu)$ where $C\in\sK$ and $\nu:C\to FC$.
Lambek's lemma states that an initial $F$-algebra is also an initial object in the category of fixed points of $F$, so the main result of \cite{goldthorpe:fixed} is proved by showing that $\Cat_{(\infty, \infty)}$ is an initial algebra for $(-)\Cat$ over presentably symmetric monoidal $\infty$-categories, and dually that $\Cat_\omega$ is a terminal coalgebra for $(-)\Cat$ over (presentably) symmetric monoidal $\infty$-categories.

Ad\'amek's construction describes a transfinite algorithm for building an initial $F$-algebra.
Specifically, if $\emptyset$ is an initial object of $\sK$, then construct the transfinite sequence
$$
\emptyset \xrightarrow! F\emptyset \xrightarrow{F(!)} F^2\emptyset \xrightarrow{F^2(!)} F^3\emptyset \to \cdots
$$
If, for some limit ordinal $\lambda$, the colimit $F^\lambda\emptyset = \varcolim_{\theta<\lambda} F^\theta\emptyset$ exists, and the canonical map $j:F^\lambda\emptyset\to F(F^\lambda\emptyset)$ is an equivalence, then $(F^\lambda\emptyset, j^{-1})$ defines an initial $F$-algebra; see \cite[Corollary 2.2.9]{goldthorpe:fixed}.
We say that Ad\'amek's construction \emph{terminates in $\lambda$ steps} in this case.

Let $\Cat_{(n, r)}$ denote the full subcategory of $\Cat_{(\infty, \infty)}$ spanned by those $\sC$ such that
\begin{itemize}
    \item parallel $k$-morphisms are equivalent for $k > n$, and
    \item all $k$-morphisms are invertible for $k > r$.
\end{itemize}
See \cite[Definition 3.0.1]{goldthorpe:fixed} or \cite[Proposition 6.1.7, Theorem 6.1.8]{gepner-haugseng} for a more precise characterisation.
We also take $\Cat_{(-1, 0)} \simeq \{\varnothing\to*\}$ and $\Cat_{(-2, 0)} \simeq \{*\}$.

The terminal object in $\SymMon_\infty$ and $\SymMon_\infty^\Pres$ is given by the singleton $\Cat_{(-2, 0)} := \{*\}$, and $\Cat_\omega$ can be realised as the limit
$$
\Cat_\omega \simeq \varlim\left(\cdots\to\Cat_{1, 3}\xrightarrow\pi \Cat_{(0, 2)} \xrightarrow\pi \Cat_{(-1, 1)} \xrightarrow\pi \Cat_{(-2, 0)}\right)
$$
which can be computed in both $\SymMon_\infty$ and $\SymMon_\infty^\Pres$.
This is precisely Ad\'amek's construction of a terminal coalgebra, proving the universal property of $\Cat_\omega$ as a terminal fixed point of enrichment in both $\SymMon_\infty$ and $\SymMon_\infty^\Pres$.

On the other hand, the initial object in $\SymMon_\infty^\Pres$ is $\Cat_{(\infty, 0)}$, and $\Cat_{(\infty, \infty)}$ can be realised as the colimit
$$
\Cat_{(\infty, \infty)} \simeq {\varcolim}^L\left(\Cat_{(\infty, 0)} \subseteq \Cat_{(\infty, 1)} \subseteq\Cat_{(\infty, 2)}\subseteq\cdots\right)
$$
in $\SymMon_\infty^\Pres$.
As this is an instance of Ad\'amek's construction of an initial algebra, this proves the universal property of $\Cat_{(\infty, \infty)}$.
However, the argument does not extend to $\SymMon_\infty$.
As shown in \cite[Proposition 3.4.1]{goldthorpe:fixed}, the first $\omega$ steps of Ad\'amek's construction for an initial algebra in $\SymMon_\infty$ yields
$$
\Cat_{<\omega} := \varcolim\left(\Cat_{(-2, 0)} \subseteq \Cat_{(-1, 1)} \subseteq \Cat_{(0, 2)} \subseteq \Cat_{(1, 3)} \subseteq \cdots\right) = \bigcup_{\substack{0\leq n<\infty \\ r \geq 0}}\Cat_{(n, r)}
$$
which is the $\infty$-category of finite-dimensional higher categories, which is not a fixed point of enrichment.

\cite[Conjecture 3.4.3]{goldthorpe:fixed} predicts that Ad\'amek's initial algebra construction never terminates, but that an initial algebra exists nonetheless.
We prove in \cref{prop:enr-init-dnt} that Ad\'amek's initial algebra construction does indeed fail to terminate.
The conjecture also characterised the initial algebra as the $\infty$-category of \emph{Noetherian} $(\infty, \infty)$-categories.
However, the definition of Noetherian in the conjecture is too weak.
We provide a modified notion of when an $(\infty, \infty)$-category is Noetherian, and prove in \cref{thm:enr-init-bigterminate} that the initial algebra for $(-)\Cat$ is precisely the full subcategory of $\Cat_{(\infty, \infty)}$ spanned by the Noetherian $(\infty, \infty)$-categories.

\section{Measuring finiteness}

To understand Ad\'amek's initial algebra construction for $(-)\Cat$ over $\SymMon_\infty$, we introduce the following measure of finiteness:
\begin{definition}\label{def:cat-rank}
    We define the \emph{rank} of an $(\infty, \infty)$-category $\sC$ by transfinite induction.
    \begin{itemize}
        \item Say that $\rank\sC < 0$ if and only if $\sC\simeq*$.
        \item For an ordinal $\theta$, say that $\rank\sC < \theta + 1$ if $\rank\Hom_\sC(x, y) < \theta$ for all $x, y\in\sC$.
        \item For a limit ordinal $\lambda$, say that $\rank\sC < \lambda$ if $\rank\sC < \theta$ for some $\theta < \lambda$.
    \end{itemize}
    Say $\rank\sC = \theta$ if $\rank\sC < \theta+1$ but $\rank\sC\not<\theta$.
    Note that the rank of $\sC$ is invariant under equivalence.

    For an ordinal $\theta$, let $\Cat_{<\theta}$ denote the full subcategory of $\Cat_{(\infty, \infty)}$ spanned by those $\sC$ with $\rank\sC < \theta$.
\end{definition}
\begin{remark}\label{rem:cat-rank}
    By \cref{lem:cat-rank} below, if $\rank\sC<\theta$ and $\theta<\theta'$, then also $\rank\sC<\theta'$.
\end{remark}
\begin{example}\label{ex:cat-rank}
    For the finite ordinal $n$, $\Cat_{<n}$ consists of the $(n-2, n)$-categories.
    
    The $\infty$-category $\Cat_{<\omega}$ consists of the finite-dimensional higher categories, and $\Cat_{<\omega+1}$ consists of the locally finite-dimensional higher categories.
\end{example}

\begin{lemma}\label{lem:cat-rank}
    The $\infty$-categories $\Cat_{<\theta}$ can be constructed through transfinite induction via enrichment:
    \begin{itemize}
        \item $\Cat_{<0} \simeq \Grpd_{-2} \simeq \{*\}$,
        \item $\Cat_{<\theta + 1} \simeq (\Cat_{<\theta})\Cat$; in particular, $\Cat_{<\theta}$ is a full subcategory of $\Cat_{<\theta+1}$,
        \item For a limit ordinal $\lambda$,
            $$
            \Cat_{<\lambda} \simeq \varcolim_{\theta<\lambda}\Cat_{<\theta}
            $$
    \end{itemize}
\end{lemma}
\begin{proof}
    That $\Cat_{<0}\simeq\{*\}$ and $\Cat_{<\theta+1} \simeq (\Cat_{<\theta})\Cat$ follow by definition.
    For the limit case, suppose by transfinite induction that $\Cat_{<\theta}\subseteq\Cat_{<\theta'}$ for all $\theta<\theta'<\lambda$.
    Then,
    \[
        \varcolim_{\theta<\lambda}\Cat_{<\theta} \simeq \bigcup_{\theta<\lambda}\Cat_{<\theta} = \Cat_{<\lambda}
    \]
    as desired.
\end{proof}
\begin{lemma}\label{lem:cat-rank-distinct}
    For every ordinal $\theta$, there is an $(\infty, \infty)$-category $\sC$ such that $\rank\sC = \theta$.
\end{lemma}
\begin{proof}
    We prove this by transfinite induction.
    For $\theta = 0$, we take $\sC = \varnothing$.
    Indeed, $\rank\sC < 1$ is vacuous, and $\rank\sC\not<0$ because $\sC\not\simeq*$.
    
    Suppose we have an $(\infty, \infty)$-category $\sD$ such that $\rank\sD = \theta$.
    Then, $\rank\sC = \theta + 1$, where $\sC := \Sigma\sD$ is the $(\infty, \infty)$-category with two objects $\bot, \top$ and $\Hom_\sC(\bot, \top) = \sD$; see \cite[Definition 4.3.21]{gepner-haugseng}.
    
    Finally, suppose $\lambda$ is a limit ordinal such that for every $\theta < \lambda$, there exists an $(\infty, \infty)$-category $\sD^\theta$ such that $\rank\sD^\theta = \theta$.
    Then, take $\sC := \coprod_{\theta<\lambda}\sD^\theta$.
    
    Let $x, y\in\sC$.
    If $x\in\sD^\theta$ and $y\in\sD^{\theta'}$ with $\theta\neq\theta'$, then $\rank\Hom_\sC(x, y) = \rank\varnothing = 0 < \lambda$.
    Otherwise, $\rank\Hom_\sC(x, y) = \rank\Hom_{\sD^\theta}(x, y) < \lambda$.
    In particular, $\rank\sC < \lambda + 1$.
    On the other hand, $\rank\sC\not<\theta$ for all $\theta < \lambda$ since $\sD^\theta$ is a (full) subcategory of $\sC$, and $\rank\sD^\theta\not<\theta$.
    Therefore, $\rank\sC \not< \lambda$, proving that $\rank\sC = \lambda$, as desired.
\end{proof}

With the above groundwork, we already have a proof of \cite[Conjecture 3.4.3(2)]{goldthorpe:fixed}.

\begin{proposition}\label{prop:enr-init-dnt}
    Ad\'amek's construction of an initial algebra for $(-)\Cat$ over the category $\SymMon_\infty$ does not terminate.
\end{proposition}
\begin{proof}
    The $\theta$th stage of Ad\'amek's construction yields $\Cat_{<\theta}$ by \cref{lem:cat-rank}.
    Therefore, the proposition follows from \cref{lem:cat-rank-distinct}.
\end{proof}

\section{Large convergence}

The failure of Ad\'amek's construction to terminate in \cref{prop:enr-init-dnt} is purely a size issue.
For instance, let $(-)\Cat^{<\omega}$ denote the subfunctor of $(-)\Cat$ that sends $\sV$ to the full subcategory $\sV\Cat^{<\omega}$ of $\sV\Cat$ spanned by those $\sV$-enriched categories with finitely many equivalence classes of objects (that is, the underlying space of objects has finitely many path-connected components).
Then, Ad\'amek's construction for $(-)\Cat^{<\omega}$ terminates after $\omega$ steps, and the initial algebra consists of those finite-dimensional higher categories with finitely many equivalence classes of $k$-morphisms for each $k\geq0$.

\begin{lemma}\label{lem:lambda-small-noeth}
    Fix a regular cardinal $\lambda$.
    Let $\sC$ be an $(\infty, \infty)$-category such that the set of equivalence classes of objects of $\sC$ is $\lambda$-small, and $\rank\Hom_\sC(x, y) < \lambda$ for all $x, y\in\sC$.
    Then, $\rank\sC < \lambda$. 
\end{lemma}
\begin{proof}
    For $x, y\in\sC$, let $\theta_{x,y} < \lambda$ such that $\rank\Hom_\sC(x, y) < \theta_{x, y}$; such an ordinal exists because a regular cardinal is necessarily a limit ordinal.
    Then, let $\theta := \sup_{x,y\in\sC}\theta_{x, y}$.
    Note that if $x\simeq x'$ and $y\simeq y'$, then $\theta_{x,y} = \theta_{x',y'}$.
    Since $\sC$ has fewer than $\lambda$ objects up to equivalence, it follows from the fact that $\lambda$ is a regular cardinal that $\theta < \lambda$, and therefore also that $\theta + 1 < \lambda$.
    Therefore, $\rank\sC < \theta + 1 < \lambda$, as desired.
\end{proof}

The above lemma enables us to prove \cite[Conjecture 3.4.3(1)]{goldthorpe:fixed}.

\begin{proposition}\label{prop:enr-init-smallterminate}
    For a regular cardinal $\lambda$, let $(-)\Cat^{<\lambda} : \SymMon_\infty \to \SymMon_\infty$ denote the subfunctor of $(-)\Cat$ that associates to a symmetric monoidal category $\sV$ the full subcategory $\sV\Cat^{<\lambda}$ of $\sV\Cat$ spanned by those $\sV$-enriched categories such that the set of path-connected components of its underlying space of objects is $\lambda$-small.
    Then, Ad\'amek's construction of an initial algebra for $(-)\Cat^{<\lambda}$ over $\SymMon_\infty$ terminates after no fewer than $\lambda$ steps.
\end{proposition}
\begin{proof}
    For an ordinal $\theta$, let $\Cat_{<\theta}^{<\lambda}$ denote the full subcategory of $\Cat_{<\theta}$ on those $(\infty, \infty)$-categories $\sC$ such that the set of equivalence classes of $k$-morphisms is $\lambda$-small for every $k\geq0$.
    Then, $\Cat_{<\theta}^{<\lambda}$ can be constructed by transfinite induction, analogous to \cref{lem:cat-rank}:
    \begin{itemize}
        \item $\Cat_{<0}^{<\lambda} \simeq \Grpd_{-2} \simeq \{*\}$, which is the initial object in $\SymMon_\infty$,
        \item $\Cat_{<\theta + 1}^{<\lambda} \simeq (\Cat_{<\theta}^{<\lambda})\Cat^{<\lambda}$,
        \item For a limit ordinal $\mu$,
            $$
            \Cat_{<\mu}^{<\lambda} \simeq \varcolim_{\theta<\mu}\Cat_{<\theta}^{<\lambda}
            $$
    \end{itemize}
    Following the proof of \cref{lem:cat-rank-distinct}, there still exists $\sC\in\Cat_{<\theta+1}^{<\lambda}$ such that $\sC\notin\Cat_{<\theta}^{<\lambda}$, so long as $\theta < \lambda$.
    However, \cref{lem:lambda-small-noeth} shows that $\Cat_{<\lambda}^{<\lambda}\subseteq\Cat_{<\theta}^{<\lambda}$ is an equivalence for all $\theta > \lambda$.
    
    Therefore, Ad\'amek's construction terminates in exactly $\lambda$ steps, as desired, and $\Cat_{<\lambda}^{<\lambda}$ carries the structure of an initial algebra for $(-)\Cat^{<\lambda}$ over $\SymMon_\infty$.
\end{proof}

It remains to prove \cite[Conjecture 3.4.3(3)]{goldthorpe:fixed}, which is concerned with the initial algebra for $(-)\Cat$ in $\SymMon_\infty$.
\begin{definition}\label{def:Noeth}
    A \emph{parallel morphism tower} $(\vec\alpha,\vec\beta)$ in an $(\infty, \infty)$-category $\sC$ is a (countable) sequence of pairs
    $$
    (\alpha_0, \beta_0), (\alpha_1, \beta_1), (\alpha_2, \beta_2), \dots
    $$
    where $\alpha_0,\beta_0$ are objects of $\sC$, and $\alpha_{n+1}$ and $\beta_{n+1}$ are parallel $(n+1)$-morphisms $\alpha_n\to\beta_n$ in $\sC$ for all $n\geq0$.
    
    Say that an $(\infty, \infty)$-category $\sC$ is \emph{Noetherian} if for any parallel morphism tower $(\vec\alpha,\vec\beta)$, there exists $N\gg0$ such that $\Hom_\sC(\alpha_N, \beta_N) \simeq *$.
    
    Denote by $\Cat_{(\infty, \infty)}^{\mathrm{Noeth}}$ the full subcategory of $\Cat_{(\infty, \infty)}$ spanned by the Noetherian $(\infty, \infty)$-categories.
\end{definition}
\begin{remark}\label{rem:Noeth}
    The above definition of Noetherian is stronger than that proposed in \cite[Definition 3.4.2]{goldthorpe:fixed}, which only requires that any parallel morphism tower $(\vec\alpha,\vec\beta)$ admits $N\gg0$ such that $\alpha_n$ and $\beta_n$ are equivalences for all $n\geq N$.
    Indeed, any $\infty$-groupoid satisfies this weaker property, but not every $\infty$-groupoid is Noetherian in the sense of \cref{def:Noeth}.
\end{remark}

\begin{lemma}\label{lem:Noeth}
    For an $(\infty, \infty)$-category $\sC$, the following are equivalent:
    \begin{enumerate}[label={(\roman*)}]
        \item\label{it:Noeth}
            $\sC$ is Noetherian,
        \item\label{it:loc-Noeth}
            $\sC$ is locally Noetherian, in the sense that $\Hom_\sC(x, y)$ is Noetherian for all $x,y\in\sC$,
        \item\label{it:small-rank}
            $\sC$ has small rank, in that $\rank\sC < \theta$ for some ordinal $\theta$,
        \item\label{it:loc-small-rank}
            $\sC$ locally has small rank, in that $\Hom_\sC(x, y)$ has small rank for all $x, y\in\sC$.
    \end{enumerate}
\end{lemma}
\begin{proof}
    The equivalence between \ref{it:Noeth} and \ref{it:loc-Noeth} follows by definition.
    
    Note that \ref{it:small-rank} certainly implies \ref{it:loc-small-rank}: if $\rank\sC < \theta$, then $\rank\sC < \theta+1$, and therefore $\rank\Hom_\sC(x, y) < \theta$ for all $x, y\in\sC$.
    Conversely, if for all $x, y\in\sC$ there exists an ordinal $\theta_{x,y}\gg0$ such that $\rank\Hom_\sC(x, y) < \theta_{x, y}$, choose $\lambda\gg0$ such that the set of equivalence classes of objects in $\sC$ is $\lambda$-small, and such that $\lambda \geq\theta_{x, y}$ for all $x, y\in\sC$.
    Then, $\rank\sC < \lambda$ by \cref{lem:lambda-small-noeth}.
    This proves that \ref{it:small-rank} is equivalent to \ref{it:loc-small-rank}.
    
    Since the singleton $*$ is certainly Noetherian, and locally Noetherian $(\infty, \infty)$-categories are Noetherian, it follows by transfinite induction on the rank that every $(\infty, \infty)$-category $\sC$ with small rank is Noetherian.
    This shows that \ref{it:small-rank} implies \ref{it:Noeth}.
    
    To prove the converse, suppose $\sC$ does not have small rank.
    Then, $\sC$ does not locally have small rank, so there must exist $\alpha_0, \beta_0\in\sC$ such that $\Hom_\sC(\alpha_0, \beta_0)$ does not have small rank.
    Proceeding recursively, we obtain a parallel morphism tower $(\vec\alpha,\vec\beta)$ where each $\Hom_\sC(\alpha_n, \beta_n)$ does not have small rank.
    In particular, $\Hom_\sC(\alpha_n, \beta_n)\not\simeq*$ for every $n\geq0$.
    Therefore, if $\sC$ does not have small rank, then $\sC$ is not Noetherian, completing the proof.
\end{proof}

\begin{theorem}\label{thm:enr-init-bigterminate}
    $\Cat_{(\infty, \infty)}^{\mathrm{Noeth}}$ carries the structure of an initial algebra for $(-)\Cat$ over $\SymMon_\infty$.
\end{theorem}
\begin{proof}
    By \cref{lem:Noeth}, the canonical inclusion $\Cat_{(\infty, \infty)}^{\mathrm{Noeth}}\subseteq(\Cat_{(\infty, \infty)}^{\mathrm{Noeth}})\Cat$ is an equivalence.

    By expanding universes, let $\Lambda$ denote the large ordinal of all (small) ordinals.
    Then, \cref{lem:Noeth} implies that $\Cat_{(\infty, \infty)}^{\mathrm{Noeth}}$ is the $\Lambda$-filtered colimit
    $$
    \Cat_{(\infty, \infty)}^{\mathrm{Noeth}} = \bigcup_\theta\Cat_{<\theta} \simeq \varcolim_{\theta<\Lambda}\Cat_{<\theta}
    $$
    in $\SymMon_\infty$, which by \cref{lem:cat-rank} is precisely $\Lambda$ stages of Ad\'amek's initial algebra construction.
    Since the construction terminates after $\Lambda$ steps by the previous discussion, the theorem follows by \cite[Corollary 2.2.9]{goldthorpe:fixed}.
\end{proof}
\begin{remark}\label{rem:enr-init-bigterminate}
    Although Ad\'amek's construction in this case requires a large colimit, this colimit is small relative to an expanded universe, and Ad\'amek's construction applies to $\infty$-categories that are small relative to the universe of discourse.
\end{remark}

\addcontentsline{toc}{section}{References}

\end{document}